\documentclass[12pt]{article}
\usepackage{latexsym}                 
\usepackage[T1]{fontenc}
\usepackage{lmodern}
\usepackage[latin1]{inputenc}
\usepackage{latexsym}                 
\usepackage{color}                 
\usepackage{amssymb}
\usepackage{amsmath}
\usepackage{amsthm}
\usepackage{graphicx}
\usepackage{hyperref}
\usepackage{verbatim}
\usepackage{mathptmx}      
\def\eref#1{(\ref{#1}%
)}
\def\RSref#1{\ref{#1}%
}
\def\RSlabel#1{\label{#1}%
}
\def\RScite#1{\cite{#1}%
}

\newcommand{\bql}[1]{%
\begin{equation}\label{#1}%
}

\def\filename#1{}
\newcommand{\eq}{\end{equation}}
\def\fa{\hbox{ for all }}

\def\dfrac#1#2{\displaystyle{\frac{#1}{#2}   }}

\def\b1{\mathbf 1}

\newcommand{\R}{\ensuremath{\mathbb{R}}}

\newcommand{\N}{\ensuremath{\mathbb{N}}}

\def\calh{{\cal H}}

\def\biglf{\par\bigskip\noindent}
\newtheorem{definition}{Definition}
\newtheorem{lemma}{Lemma}

\newtheorem{theorem}{Theorem}
\newtheorem{corollary}{Corollary} 
\newtheorem{example}{Example}
\def\OOmega{\overline{\Omega}}
\begin{document}
\begin{center}
{\Large \bf Superconvergence of Kernel-Based Interpolation}
\vspace{0.5cm}

Robert Schaback

\vspace{0.5cm}

Draft of \today
\end{center}
{\bf Abstract}: 
From spline theory it is well-known that univariate 
cubic spline
interpolation, if carried out in its natural Hilbert space $W_2^2[a,b]$
and on point sets with fill distance $h$, 
converges only like ${\cal O}(h^2)$ in $L_2[a,b]$ if no additional
assumptions are made. But {\em superconvergence} up to order $h^4$ 
occurs if more smoothness is assumed and if certain 
additional boundary conditions
are satisfied. This phenomenon was generalized in 1999
to multivariate interpolation 
in Reproducing Kernel Hilbert Spaces on domains $\Omega\subset\R^d$ 
for continuous positive definite Fourier-transformable shift-invariant
kernels on $\R^d$.  But the sufficient condition for
superconvergence given in 1999 still needs further analysis, because
the interplay between smoothness and boundary conditions is not clear at all.
Furthermore, if only additional smoothness is assumed, superconvergence
is numerically observed in the interior of the domain, but without
explanation, so far.  
This paper first  
generalizes the ``improved error bounds''  of 1999 by an
abstract theory that includes the Aubin-Nitsche trick and
the known superconvergence results for univariate polynomial splines.
Then the paper 
analyzes what is behind the sufficient conditions for superconvergence. 
They split into conditions on {\em smoothness} and 
{\em localization},
and these are investigated independently.  
If sufficient smoothness is present,
but no additional localization conditions are assumed, it is proven that
superconvergence always occurs in the interior of the domain.
If smoothness and localization interact in the kernel-based 
case on $\R^d$, weak and strong boundary conditions
in terms of pseudodifferential operators occur. 
A special section on Mercer expansions
is added, because Mercer eigenfunctions always satisfy 
the  sufficient conditions for superconvergence.
Numerical examples illustrate the theoretical findings.
\section{Introduction}\RSlabel{SecIntro}
This paper investigates the superconvergence phenomenon in detail,
using the term ``superconvergence'' for a situation where
the approximating functions (approximants) have less smoothness
than the approximated function (the approximand), while the 
smoothness of the latter determines the error bound and the 
convergence rate. This is well-known from univariate spline theory 
\RScite{ahlberg-et-al:1967-1,nuernberger:1989-1,schumaker:2007-1} and
the Aubin-Nitsche trick in finite elements \RScite{braess:2001-1}.
Other notions of superconvergence, mainly in finite elements
\RScite{bramble-schatz:1977-1, thomee:1977-1, wahlbin:1995-1} 
refer to higher-order convergence in special
points like vertices of a refined triangulation. 
Superconvergence in the sense  of this paper occurs 
in the whole domain or in a subdomain. In contrast
to the ``escape'' situation of 
\RScite{narcowich-et-al:2006-1}, where smoothness of the approximands
is lower than the smoothness of the approximants,
we consider the case where
smoothness of the approximands is higher. In \RScite{narcowich-et-al:2006-1},
the convergence rate is like the one 
for the kernel of the larger space with less smoothness, while here
the convergence rate is equal to the rate obtainable using
the smoother kernel of a smaller space.  
\biglf
The paper starts with a unified abstract presentation of 
the standard cases of superconvergence, including finite elements,
splines, sequence spaces, and kernel-based interpolation on domains in $\R^d$. 
The sufficient criterion
for superconvergence in the abstract situation splits 
into two conditions in Section \RSref{SecLoc} as soon as {\em localization} 
comes into play. In Section \RSref{SecFTS}, the paper 
specializes to kernel-based function spaces on bounded domains in $\R^d$,
linking localization to weak and strong solutions of homogeneous
pseudodifferential equations outside the domain. In the Sobolev case
$W_2^m(\R^d)$ treated in Section \eref{SecSob}, the operators  
are classical, namely $(Id-\Delta)^m$, and hidden boundary conditions 
come finally into play, namely when a general function $f$ 
on $\Omega$ with extended smoothness $W_2^{2m}(\Omega)$ is considered.
Superconvergence then requires that $f$ has an 
extension to $\R^d$ by solutions
of $(Id-\Delta)^m=0$ with $W_2^{2m}(\R^d)$ smoothness, and this 
imposes the condition $(Id-\Delta)^m=0$ in
the $W_2^{2m}(\R^d)$ sense on the boundary. Then Section \RSref{SecIntSupConv}
applies the previous results to show that superconvergence always
occurs in the interior of the domain, if the approximants have 
sufficient smoothness. 
\biglf
Because Mercer expansions of continuous kernels
yield local eigenfunctions satisfying the criteria for superconvergence,
Section \RSref{SecMExt} links the previous localization 
and extension results to Mercer expansions. In particular, the 
Hilbert space closure of the extended Mercer eigenfunctions
coincides with the closure of all possible interpolants with nodes in the
domain.  
Numerical examples in Section \RSref{SecExa}
illustrate the theoretical results,
in particular demonstrating the superconvergence in the 
interior of the domain. 
\section{Abstract Approach}\RSlabel{SecAbAp}
The basic argument behind superconvergence in the sense of this paper
has a very simple abstract form that works for univariate splines, 
finite elements, and kernel-based methods. To align it with
what follows later, we use a somewhat special notation.
\biglf
The starting point
is a Hilbert space $\calh_K$ with inner product $(.,.)_K$
and a linear best approximation problem in the norm of $\calh_K$ 
that can be described by a projector
${\Pi_K}$ from  $\calh_K$ onto a closed subspace ${\Pi_K}(\calh_K)$. 
The standard error analysis of such a process 
uses a weaker norm $\|.\|_0$ that we assume to 
arise from a Hilbert space $\calh_0$ with continuous embedding
$E^K_0\;:\;\calh_K\to\calh_0$. It takes the form
\bql{eqstanderrbnd}
\|E^K_0(f-{\Pi_K} f)\|_0\leq \epsilon
\|f-{\Pi_K} f\|_K \fa f\in\calh_K
\eq
and usually describes standard convergence results when the projectors vary.
\begin{theorem}\RSlabel{theGenSupConv}
Superconvergence occurs in the subspace
$\calh_{K*K,0}:=(E_0^K)^*(\calh_0)$ of $\calh_K$ 
and turns a standard error bound \eref{eqstanderrbnd} into
$$
\|E^K_0(f-{\Pi_K} f)\|_0\leq \epsilon^2
\|((E_0^K)^*)^{-1}f\|_0 \fa f\in\calh_{K*K,0}.
$$
\end{theorem} 
\begin{proof}\RSlabel{ProtheGenSupConv}
If $f=(E_0^K)^*(v_f)$ with $v_f\in\calh_0$, then
\bql{eqswapLL}
(f,g)_K
=
((E_0^K)^*(v_f),g)_K
=
(v_f,E_0^Kg)_K \fa g\in \calh_K,\;f\in \calh_{K*K,0}
\eq
and we get via orthogonality
$$
\begin{array}{rcl}
\|f-{\Pi_K} f\|^2_K
&=&
(f, f-{\Pi_K} f)_K\\
&=&
((E_0^K)^*(v_f), f-{\Pi_K} f)_K\\
&=&
(v_f, E_0^K(f-{\Pi_K} f))_0\\
&\leq &
\|v_f\|_0\|E_0^K(f-{\Pi_K} f)\|_0\\
&\leq &
\epsilon \|v_f\|_0
\|f-{\Pi_K} f\|_K\\
\end{array}
$$
leads to the assertion.
\end{proof} 
\begin{example}\RSlabel{exaFEM}
The Aubin-Nitsche trick in finite elements takes
the spaces $\calh_K = H_0^1(\Omega)\subset \calh_0=L_2(\Omega)$
and uses the fact that piecewise linear finite elements are 
best approximations in  $H_0^1(\Omega)$. 
The standard ${\cal O}(h)$
convergence rate in $H_0^1(\Omega)$ leads to superconvergence 
of order $h^2$ in $\calh_{K*K,0}=H^2(\Omega)\cap H_0^1(\Omega)$, 
though the approximants
do not lie in that space.  The condition \eref{eqswapLL} is
$$
\begin{array}{rcl}
(f,g)_{K}
&=&
(\nabla f, \nabla g)_{L_2(\Omega)}\\
&=&
(-\Delta f, E^K_0 g)_0\\
&=&
(v_f, E^K_0 g)_0 \fa g\in \calh_K=H_0^1(\Omega),
\end{array}
$$
but note that vanishing boundary values are important here. 
\end{example} 
\begin{example}\RSlabel{exaUniSpl}
In basic univariate spline theory \RScite{ahlberg-et-al:1967-1,%
nuernberger:1989-1,schumaker:2007-1}
for splines of order $2n$ or degree $2n-1$, the spaces 
are $\calh_0=L_2[a,b]$ and $\calh_K=W_2^n[a,b]$,
but a seminorm is used there. The projector
is interpolation on finite point sets, and it has the orthogonality property
because it is minimizing the proper seminorm. Then the 
abstract condition \eref{eqswapLL} is treated like
$$
\begin{array}{rcl}
(f,g)_{K}
&=&
(D^n f, D^ng)_{L_2(\Omega)}\\
&=&
((-1)^nD^{2n}f, E^K_0g)_0\\
&=&
(v_f, E^K_0g)_0 \fa g\in \calh_K,
\end{array}
$$
but note that it requires 
certain boundary conditions to be satisfied that we do not 
consider in detail here.
\end{example} 
These two examples show that \eref{eqswapLL} may contain hidden boundary
conditions, but these are not directly connected to superconvergence.  
They concern the transition from the second to the third formula  in 
\eref{eqswapLL}.  But we shall see now that \eref{eqswapLL}
may hold without boundary conditions:
\begin{example}\RSlabel{exaSeqSpe}
For kernels with series expansions like Mercer kernels,
the basic theory boils down to sequence spaces starting from
$\calh_0=\ell_2(\N)$. For arbitrary positive sequences
$\kappa:=\{\kappa_n\}_n$ with $\lim_{n\to\infty} \kappa_n=0$, the
Hilbert space $\calh_K$ is defined via sequences $f=\{f_n\}_n,\;
g=\{g_n\}_n$ 
with
$$
(f,g)_K:=\displaystyle{\sum_n \frac{f_ng_n}{\kappa_n}   } 
$$  
to contain all $f$ with $\|f\|_K<\infty$. Projectors
$\Pi_K\;:\;\calh_K\to\calh_K $ 
should be norm-minimizing, e.g. as projectors on subspaces.
Then \eref{eqswapLL} is
$$
(f,g)_K
= \displaystyle{\sum_n \frac{f_n}{\kappa_n}g_n   }
=(f./\kappa,g)_0=(v_f,g)_0 
$$ 
in MATLAB notation, and we see that $H_{K*K,0}$ is the space 
generated by the sequence $\kappa.*\kappa$ in MATLAB notation. 
There is no localization like \eref{eqlocaliz}, and 
there cannot be any hidden ``boundary conditions''.
It is easy to apply this to analytic cases with series expansions,
e.g. into orthogonal polynomials or spherical harmonics. 
\end{example}  
This example explains our seemingly strange notation 
in the abstract setting,
but the most important case is still to follow:
\begin{example}\RSlabel{exaKerCase}
For dealing with the multivariate kernel-based case
in \RScite{schaback:1999-1}, we take a (strictly) positive definite 
translation-invariant, continuous, and  
Fourier-transformable kernel $K$ on $\R^d$ 
to define $\calh_K$ as the {\em native} Hilbert space in which it is
reproducing. For a bounded domain $\Omega$ with an interior cone condition,
we use $\calh_0=L_2(\Omega)$ and have a continuous embedding.
Sampling inequalities \RScite{rieger:2008-1,rieger-et-al:2010-1} 
yield standard error bounds
\eref{eqstanderrbnd}. The 
abstract condition \eref{eqswapLL} is now treated via
$$
\begin{array}{rcl}
(f,g)_{K}
&=&
\displaystyle{\int_{\R^d}\frac{\hat f \overline{\hat g}}{\hat K}   }\\
&=&
\displaystyle{\int_{\R^d}\frac{\hat f }{\hat K}\overline{\hat g}   }\\
&=&
\left((\frac{\hat f }{\hat K})^\vee , E^K_0g  \right)_{L_2(\R^d)}\\
&=&
(v_f, E^K_0g)_{L_2(\Omega)} \fa g\in \calh_K,
\end{array}
$$
if we assume
\bql{eqconvol}
f=K*v_f \hbox{ with } v_f\in L_2(\R^d)
\eq 
and
\bql{eqlocaliz}
 v_f\in L_2(\R^d) \hbox{ supported in } \Omega.  
\eq 
The space of functions with the
{\em convolution condition} \eref{eqconvol} is $\calh_{K*K}$
where the convolved kernel $K*K$ is reproducing, and the additional 
{\em localization} condition \eref{eqlocaliz}  defines
a subspace $\calh_{K*K,0}$ that we shall study in more detail 
in the rest of the paper. Since Fourier transform tools
require global spaces like  $L_2(\R^d)$ or $W_2^m(\R^d)$ while
error bounds only work on local spaces like $L_2(\Omega)$ or $W_2^m(\Omega)$,
we have to deal with {\em localization}, and in particular 
we must be very careful with maps that restrict or extend functions
between these spaces.  
\end{example} 
We first handle localization by a small add-on to the abstract theory. 
In contrast to the setting above, we use spaces $\calh_0$ and $\calh_K$ 
that do not need localization, i.e. they stand for 
$L_2(\R^d)$ or $W_2^m(\R^d)$. Then we add an abstract {\em localized} space   
$\calh_\Omega$ standing for $L_2(\Omega)$ 
with
additional maps  $E_\Omega^0\;:\;\calh_0\to \calh_\Omega$
and vice versa, modelling restriction to $\Omega$ and extension by zero.
Throughout, we shall use a ``cancellation'' 
notation for embeddings, allowing e.g. $E_A^BE_B^C=E_A^C$.
These maps should have the properties
\bql{eqaddproploc}
\begin{array}{rcl}
(E_0^\Omega f, E^\Omega_0 g)_0&=&(f, g)_\Omega \fa f,g\in \calh_\Omega,\\
(f, E^\Omega_0g)_0&=& (E_\Omega^0f,g)_\Omega \fa  f\in \calh_0,\;g\in\calh_\Omega.
\end{array}
\eq
To generalize the splitting of the abstract condition
\eref{eqswapLL} into the {\em convolution} condition
\eref{eqconvol} and the {\em localization} condition
\eref{eqlocaliz}, we postulate  
\bql{eqconvolgen}
(f,g)_K=(v_f,E_0^Kg)_0 \fa f\in \calh_{K*K}:=(E_0^K)^*(\calh_0)
\eq
without localization, and then define $\calh_{K*K,\Omega}$ as the
subspace of $\calh_{K*K}$ of all $f\in \calh_{K*K}$ with
\bql{eqlocalizgen}
v_f=E_0^\Omega E_\Omega^0 v_f,
\eq
caring for localization. 
\begin{theorem}\RSlabel{theSupConvLoc}
Besides \eref{eqaddproploc}, \eref{eqconvolgen}, and \eref{eqlocalizgen},
assume a partially 
localized error bound of the form
\bql{eqlocstabnd}
\|E^0_\Omega E^K_0(f-{\Pi_K} f)\|_\Omega
\leq \epsilon \|f-{\Pi_K} f\|_K \fa f\in\calh_K
\eq
describing a standard convergence behavior,
where the constant $\epsilon$ now also depends on $\Omega$. 
Then for all $f\in \calh_{K*K,\Omega}$ we have superconvergence in the sense  
$$
\|E^0_\Omega E^K_0(f-{\Pi_K} f)\|_\Omega
\leq \epsilon^2\|v_f\|_0.
$$
\end{theorem} 
\begin{proof}\RSlabel{ProtheSupConvLoc}
We change the start of the basic argument to 
$$ 
\begin{array}{rcl}
\|E^0_\Omega E^K_0(f-{\Pi_K} f)\|^2_\Omega
&\leq &
\epsilon^2\|f-{\Pi_K} f\|^2_K\\
&=&
\epsilon^2(v_f,E^K_0(f-{\Pi_K} f))_0
\end{array}
$$ 
and then have to introduce a localization in the right-hand side as well.
This works by the additional assumptions \eref{eqconvolgen}
and \eref{eqlocalizgen} and yields
$$
\begin{array}{rcl}
\|E^0_\Omega E^K_0(f-{\Pi_K} f)\|^2_\Omega
&\leq &
\epsilon^2
(E^K_0(f-{\Pi_K} f), E_0^\Omega E_\Omega^0 v_f)_0\\
&=&
\epsilon^2
(E_\Omega ^0 E^K_0(f-{\Pi_K} f), E_\Omega^0 v_f)_\Omega\\
&\leq &
\epsilon^2
\|E_\Omega ^0 E^K_0(f-{\Pi_K} f)\|_\Omega \|E_\Omega^0 v_f\|_\Omega\\
&=&
\epsilon^2
\|E_\Omega ^0 E^K_0(f-{\Pi_K} f)\|_\Omega \|v_f\|_0.\\
\end{array}
$$ 
\end{proof}   
Summarizing, we see that the abstract condition \eref{eqswapLL} 
contains localization and boundary conditions in the first two examples,
while the third is completely 
without these conditions, and the fourth contains localization, but
no boundary condition.
This strange fact needs clarification.
Another observation in the kernel-based multivariate case 
of Example \RSref{exaKerCase}  is that
additional smoothness in the sense of \eref{eqconvolgen} leads to
superconvergence in the interior of the domain, even in cases where
\eref{eqlocalizgen} does not hold.  
We shall focus on these items from now on.
\section{Localization}\RSlabel{SecLoc}
We now come back to the second part of the abstract
theory in Section \RSref{SecAbAp} and have a closer look
at {\em localization}. The {\em localized} space $\calh_\Omega$
still is separated from the ``global'' spaces $\calh_K$ and $\calh_0$,
but we now push the localization into subspaces of   $\calh_K$.
To this end, consider the orthogonal closed subspaces 
\bql{eqZKO}
Z_K(\Omega)= \ker E_\Omega^0E_0^K \hbox{ and }
\calh_K(\Omega):=Z_K(\Omega)^\perp = (E_\Omega^0E_0^K)^*(\calh_\Omega)
\eq
of $\calh_K$. 
The second space consists of all ``functions'' $f$ 
in $\calh_K$ that are completely determined by  their ``values
on $\Omega$'', i.e. by  $E_\Omega^0E_0^Kf$. 
This is the space
users work in  when they take spans of linear combinations
of kernel translates $K(\cdot,x)$ with $x\in\Omega$. The 
orthogonal complement of the $\calh_K$-closure then consists
of all functions in $\calh_K$ that vanish on $\Omega$, i.e. it is
$Z_K(\Omega)$ in the above decomposition. 
\biglf
To make this more explicit, 
recall the native space construction
for continuous (strictly) positive definite kernels
on $\R^d$ starting from arbitrary 
finite sets $X=\{x_1,\ldots,x_N\}\subset\R^d$ and 
weight vectors $a\in\R^N$. These are used to define the generators 
\bql{eqNSgen}
\mu_{X,a}(f):=\displaystyle{\sum_{j=1}^Na_jf(x_j)   },\;
f_{X,a}(x):=\displaystyle{\sum_{j=1}^Na_jK(x_j,x)   } 
\eq
for the native space construction, and they are connected by the Riesz map.
One defines inner products on the generators via kernel matrices and then
goes to the Hilbert space closure to get $\calh_K$.  
\biglf
If the sets are restricted to a domain $\Omega$, the same process applies
and yields a closed subspace 
$\calh(K,\Omega)$ of $\calh_K$ that we might call
a {\em localization} of $\calh_K$. 
It is that subspace in which
standard kernel-based methods work, using point sets that always lie in
$\Omega$.
\begin{lemma}\RSlabel{lemStaLocOrth}
The subspace $\calh(K,\Omega)$ of $\calh_K$ defined above 
coincides with the space $\calh_K(\Omega)$ defined
abstractly above.
The isometric embedding  $\calh_K(\Omega)\to \calh_K$ 
maps each function in $\calh_K(\Omega)$ to the unique
$\calh_K$-norm-minimal extension to $\R^d$.  
\end{lemma}   
\begin{proof}\RSlabel{ProlemStaLocOrth}
The reproduction property $\mu_{X,a}(f)=(f,f_{X,a})_{K}$ immediately yields
the first statement, because the spanned space is the orthogonal
complement of $Z_K(\Omega)$ of \eref{eqZKO}.
The second follows from the variational fact that
any norm-minimal extension must be $\calh_K$-orthogonal to all
functions in $\calh_K$ that vanish on $\Omega$.
\end{proof} 
Before we go further, we could say that a function $f\in\calh_K$
can be {\em localized} to $\Omega$, if it lies in $\calh_K(\Omega)$.
And, we could define the $K$-{\em carrier} of 
$f\in\calh_K$ as the smallest
domain that $f$ can be localized to, i.e. the closed set $\Omega_f$ such that
$\calh_K(\Omega_f)$ is the intersection of all $\calh_K(\Omega)$ 
such that $f$ can be
localized to $\Omega$. It is an interesting problem to find the carrier
of functions in $\calh_K$, and we shall come back to it.
\biglf
After this detour explaining $\calh_K(\Omega)$, 
we assume that the range of the projector $\Pi_K$ 
is in $\calh_K(\Omega)$ and thus orthogonal to $Z_K(\Omega)$.
The standard approach to working with $\R^d$-kernels
on domains $\Omega$ starts with $\calh_\Omega$ right away and does not care
for $\calh_K=\calh_{\R^d}$.  These spaces are norm-equivalent, but not the same.
They are connected by extension and restriction maps like above.
\begin{lemma}\RSlabel{lemNoSupCon}
If $f\in\calh_K$ is not in $\calh_K(\Omega)$, the superconvergence argument
fails already in \eref{eqlocstabnd},
because  there is a positive constant $\delta$ depending on $f,\;K,$ 
and $\Omega$, but not on $\Pi_K$, such that
$$
\|f-\Pi_K f\|_K\geq \delta.
$$
\end{lemma}  
\begin{proof}\RSlabel{ProlemNoSupCon}
This is clear because the left-hand side can never be smaller than the
norm of the best approximation to $f$ from 
the closed subspace $\calh_K(\Omega)$.
\end{proof} 
Note that the above argument does not need extended smoothness. 
But with extended smoothness, we get
\begin{lemma}\RSlabel{lemneccond}
The sufficient 
conditions \eref{eqconvolgen} and \eref{eqlocalizgen} 
for superconvergence imply $f\in\calh_K(\Omega)$. 
\end{lemma} 
\begin{proof}\RSlabel{Prolemneccond}
For $f\in\calh_K$ satisfying both conditions, and any $w\in \calh_K$
we get
\bql{eqfwchain}
\begin{array}{rcl}
(f,w)_K
&=&
(v_f, E^K_0w)_0\\
&=&
(E_0^\Omega E_\Omega^0 v_f, E^K_0w)_0\\
&=&
(E_\Omega^0 v_f, E^0_\Omega E^K_0w)_\Omega\\
\end{array} 
\eq
and this vanishes for $w\in Z_K(\Omega)$.
\end{proof}
\begin{theorem}\RSlabel{theEqui}
The conditions 
\eref{eqconvolgen} and \eref{eqlocalizgen} are equivalent to
\bql{eqaddass2}
f\in \calh_K(\Omega) \hbox{ and } f\in \calh_{K*K}
\eq
if $\calh_K$ is dense in $\calh_0$.
\end{theorem} 
\begin{proof}\RSlabel{ProtheEqui}
We only have to prove that the above conditions yield
\eref{eqlocalizgen}. The conditions imply
that there must be some $f^\Omega\in\calh_\Omega$ such that
$$
(f,w)_K
=
(v_f, E^K_0w)_0=(f^\Omega,E^0_\Omega E^K_0w)_\Omega
$$
for all $w\in\calh_K$. But then
$$ 
\begin{array}{rcl}
(v_f, E^K_0w)_0
&=&
(E_0^\Omega f^\Omega, E^K_0w)_\Omega\\
\end{array} 
$$
and by density we get $v_f=E_0^\Omega f^\Omega$ and 
$f^\Omega=E_\Omega^0 v_f$ and 
$$
v_f=E_0^\Omega f^\Omega=E_0^\Omega E_\Omega^0 v_f.
$$
\end{proof} 
The advantage of \eref{eqaddass2} 
is that the two conditions for 
smoothness and localization are decoupled, i.e. $\calh_K(\Omega)$
does not refer to $K*K$ in any way.
\biglf 
Two things are left to do: if we only assume smoothness,
i.e. $f\in\calh_{K*K}$, we should get superconvergence in the interior
of the domain, and the conditions \eref{eqaddass2} should contain
a hidden boundary condition. The examples \RSref{exaFEM} and \RSref{exaUniSpl} 
use differential operators explicitly, while Example \RSref{exaKerCase}
has pseudodifferential operators in the background. Therefore 
the next section adds details to Example
\RSref{exaKerCase}, building on the abstract results
of Sections \RSref{SecAbAp} and \RSref{SecLoc}.
\section{Fourier Transform Spaces}\RSlabel{SecFTS}
By $\calh_K$ we 
denote the global Hilbert space on $\R^d$
generated by a translation-invariant Fourier-transformable
(strictly) positive definite kernel $K$ with strictly positive
Fourier transform $\hat{K}$, and the inner product will be
denoted by $(.,.)_K$ for simplicity. 
For elements $f,\,g\in \calh_K$ the
inner product in Fourier representation is
\bql{eqFouRep}
(f,g)_K=\displaystyle{\int_{\R^d}\frac{\hat f(\omega)\overline{\hat g(\omega)}}%
{\hat K(\omega)}d\omega   }
\eq
where we ignore the correct multipliers for simplicity, 
even though we later
use Parseval's identity.
We can rewrite this as
\bql{eqKLK}
\begin{array}{rcl}
(f,g)_K&=&\displaystyle{\int_{\R^d}
\frac{\hat f(\omega)}{\sqrt{\hat K(\omega)}}
\frac{\overline{\hat g(\omega)}}%
{\sqrt{\hat K(\omega)}}d\omega   }\\
&=& (L_K(f),L_K(g))_{L_2(\R^d)}
\end{array} 
\eq
with the standard isometry $L_K\;:\;\calh_K\to \calh_0:=L_2(\R^d)$ defined by
$$
L_K(f)=\left(\frac{\hat f}{\sqrt{\hat K}}\right)^\vee
$$ 
and the somewhat sloppy convolution notation
\bql{eqslopcon}
f=L_K(f)*\sqrt{K}
\eq
involving the {\em convolution-root} of $K$, i.e. the kernel with
\bql{eqKKK}
(\sqrt{K})^\wedge(\omega)=\sqrt{\hat K(\omega)} \fa \omega\in\R^d
\eq
such that $K=\sqrt{K}*\sqrt{K}$. 
\biglf
In a similar way we define $\calh_{K*K}$ and $L_{K*K}$ 
to get $v_f=L_{K*K}f$ by \eref{eqconvol}.
In case of $g=K(x,\cdot)$ in \eref{eqconvolgen}, we have 
\bql{eqf2KK} 
\begin{array}{rcl}
f(x)&=&
(f,K(x,\cdot))_K\\
&=&
(L_{K*K}(f),K(x,\cdot))_{L_2(\R^d)}\\
&=&
(f,L_{K*K}K(x,\cdot))_{L_2(\R^d)}
\end{array}
\eq
under certain additional conditions.
The second line allows to recover particular solutions of the equation
$L_{K*K}f=g$ for sufficiently smooth $f$, while the standard use of the third
is connected to $K(x,\cdot)$ being a fundamental solution to that equation. 
Both cases arise very frequently in papers that solve partial differential
equations via kernels, using fundamental or particular solutions.
See e.g. \RScite{li-et-al:2010-1} for short survey of both,
with many references.
\biglf
For Theorem \RSref{theEqui} we need that $\calh_K$ is dense 
in $\calh_0=L_2(\R^d)$. By a simple Fourier transform argument, any
$f\in \calh_0=L_2(\R^d)$ that is orthogonal to all functions
in $\calh_K$ must have the property $\hat f \cdot \sqrt{\hat K}=0$ 
almost everywhere,
and thus $f=0$ in $L_2$. 
\biglf
In the Fourier transform situation,
the extension of a function $f\in\calh_K(\Omega)$ to a global function
already contains a hidden boundary condition that does not explicitly appear
in practice. For any $f\in\calh_K(\Omega)$ there is a function $f_\Omega\in
\calh_\Omega=L_2(\Omega)$
such that $f=(E_\Omega^0 E^K_0)^* f_\Omega$, i.e.
$$
\begin{array}{rcl}
(f,v)_K&=&(L_Kf,L_Kv)_{L_2(\R^d)}\\
&=&
(f_\Omega, E_\Omega^0 E^K_0v)_{L_2(\Omega)} \fa v\in \calh_K.
\end{array} 
$$
We can split $\calh_0=L_2(\R^d)$ for any domain
$\Omega$ into a direct orthogonal sum of 
$\calh_\Omega$ and $\calh_{\overline{\Omega}}$, the domain
$\overline{\Omega}$ being the closure of the complement of $\Omega$.
Then 
\bql{eqomegasplit}
\begin{array}{rcl}
0&=&(E_\Omega^0 L_Kf-f_\Omega,E_\Omega^0 L_Kv)_{L_2(\Omega)}\\
0&=& (E_{\OOmega}^0 
  L_Kf,E_{\OOmega}^0  L_Kv)_{L_2(\OOmega)}\\
\end{array} 
\eq
for all  $v\in \calh_K$. If we have additional smoothness in the sense
$f\in\calh_{K*K}$, then
$$
\begin{array}{rcl}
(f,v)_K
&=&
(L_{K*K}f,E^K_0v)_{L_2(\R^d)}=(f_\Omega, E_\Omega^0 E^K_0v)_{L_2(\Omega)}\\
\end{array}
$$
implies $f_\Omega=E^0_\Omega L_{K*K}f$ and $0=E^0_{\OOmega} L_{K*K}f$. i.e.
the equation $L_{K*K}f=0$ holds in $\OOmega$. This motivates 
\begin{definition}\RSlabel{defWS}
If $f\in \calh_K$ satisfies the second equation of \eref{eqomegasplit} 
for all $v\in \calh_K$, we say that $f$ is a $\calh_K$-weak solution
of $L_{K*K}f=0$ in $\OOmega$. 
\end{definition} 
\begin{theorem}\RSlabel{theWeakBC}
The functions $f\in \calh_K(\Omega)$ are $\calh_K$-weak
solutions of $L_{K*K}f=0$ on $\OOmega$. \qed
\end{theorem} 
In a somewhat sloppy formulation, the functions $f\in\calh_K(\Omega)$
are extended to $\calh_K(\R^d)$ by $\calh_K$-weak solutions
of $L_{K*K}f=0$
outside $\Omega$. 
\begin{corollary}\RSlabel{corStrongCase}
The functions $f\in \calh_{K*K}\cap \calh_K(\Omega)$,
i.e. those with superconvergence, are strong solutions
of $L_{K*K}f=v$ in $\R^d$ with a function $v\in L_2(\Omega)$ extended by zero to
$\R^d$. \qed
\end{corollary} 
\section{The Sobolev Case}\RSlabel{SecSob}
Our main example is Sobolev space $W_2^m(\R^d)$ with the exponentially decaying
Whittle-Mat\'ern kernel
$$
W_{m,d}(r)=r^{m-d/2}K_{m-d/2}(r),\;r=\|x-y\|_2,\;x,\,y\in\R^d
$$
written in radial form using the modified Bessel function $K_{m-d/2}$ 
of second kind. We use the notation $K$ for kernels differently 
elsewhere. 
\biglf
For the kernel $K=W_{m,d}$, the inverse of the mapping 
$L_{K*K}=L_{W_{2m,d}}\;:\;W_2^{2m}(\R^d)\to L_2(\R^d)$ is the convolution with
the kernel ${K}=W_{m,d}$, and thus $L_{K*K}$ 
coincides with the differential operator
$(Id-\Delta)^m$ that has the generalized Fourier transform 
$(1+\|\omega\|_2^2)^{m}$. Now Theorem \RSref{theWeakBC} implies that
all $f\in \calh_K(\Omega)$ are  $W_2^{m}(\R^d)$-weak solutions of
the partial differential equation $(Id-\Delta)^m f=0$ outside $\Omega$,
while Corollary \RSref{corStrongCase} implies that functions
$f\in \calh_{K*K}\cap \calh_K(\Omega)$ are strong solutions. 
Conversely, the functions $f\in\calh_K(\Omega)$
are extended to $\calh_K(\R^d)$ by weak solutions
of $(Id-\Delta)^m f=0$
outside $\Omega$ that satisfy boundary conditions at infinity
and on $\partial\Omega$  
to ensure $f\in\calh_K$. Since the functions in $\calh_K(\Omega)$
and $W_2^m(\Omega)$ are the same, the extension over $\partial\Omega$ 
is always possible and poses no restrictions to functions in $\calh_K(\Omega)$.
\begin{example}\RSlabel{exaW21}
As an illustration, consider $\calh_K=W_2^2(\R)$ with the radial
kernel $(1+r)\exp(-r)$ up to a constant factor. Solutions
of $L_4f:=(f-f'')-(f-f'')''=0$  are linear combinations of
$e^x, \,xe^x, \,e^{-x},\, xe^{-x}$, and for $\Omega=[a,b]$ we
see that functions $f\in W_2^2[a,b]$ are extended for $x\leq a$ by 
linear combinations of
$e^x$ and $xe^x$ only, while for $x\geq b$ one has to take
the basis $e^{-x},\, xe^{-x}$ to have the
extended function in  $\calh_K=W_2^2(\R)$. This poses no additional
constraints for functions in  $W_2^2[a,b]$, because only $C^1$ continuity
is necessary, and the extensions are unique.
\biglf
Similarly, functions 
$f\in \calh_{K*K}\cap \calh_K(\Omega)=W_2^{4}(\R)\cap W_2^2[a,b]$ 
are strong solutions of $L_4f=0$ outside $[a,b]$ with full
$W_2^{4}(\R)$ continuity over the boundary. Here, the hidden boundary
conditions creep in when one starts with arbitrary functions
from $W_2^4[a,b]$. Not all of these have $W_2^{4}(\R)$-continuous
extensions to solutions of $L_4f=0$ outside $[a,b]$, because we now
need $C^3$ smooth transitions to the span of $e^x$ and $xe^x$ for $x\leq a$
and to $e^{-x},\, xe^{-x}$ for $x\geq b$. An explicit calculation
yields the necessary boundary conditions
$$
f(a)=f'(a)=f''(a)=f'''(a),\;f(b)=-f'(b)=f''(b)=-f'''(b).
$$
We come back to the example in Section \RSref{SecExa}.
\end{example} 
In general, the exterior problem $(Id-\Delta)^m f=0$
outside $\Omega$ is always weakly uniquely solvable for boundary conditions
coming from a function $f\in W_2^m(\Omega)$, the solution being
obtainable by the standard kernel-based extension. This is 
no miracle, because $K(x,\cdot)$ is the fundamental
solution of  $(Id-\Delta)^m =0$ at $x$ in the sense of 
Partial Differential Equations, and superpositions
of such functions with $x\in\Omega$ will always satisfy $(Id-\Delta)^m =0$ 
outside $\Omega$. 
\biglf
However, strong solutions of  $(Id-\Delta)^m =0$ outside
$\Omega$ with $W_2^{2m}(\R^d)$ regularity will not necessarily exist 
as extensions of arbitrary functions in $W_2^{2m}(\Omega)$,
as the above example explicitly shows. This is no objection
to the fact that all such functions have extensions to $\R^d$ with
$W_2^{2m}(\R^d)$ regularity, but not all of these extensions are 
in $\calh_K(\Omega)$ to provide superconvergence.  
\begin{example}\RSlabel{exaWend1}
The compactly supported Wendland kernels \RScite{wendland:1995-1} 
are reproducing in Hilbert spaces that are norm-equivalent to Sobolev spaces,
but their associated pseudodifferential operators $L_{K*K}$
with symbols $\hat{K}^{-1}$ are somewhat messy
because their Fourier transforms \RScite{chernih-hubbert:2014-1} are.
Nevertheless, the kernel  translate $K(x,\cdot)$ is a fundamental solution
of $L_{K*K}f=0$ at $x$, and the fundamental solutions have the nice property
of compact support. Further details are left open.
\end{example}
\begin{example}\RSlabel{exaGau}
For other situations with pointwise meaningful pseudodifferential operators
like in the Gaussian case with 
$$
L_{K*K}f=\displaystyle{\sum_{n=0}^\infty \frac{(-\Delta)^n f}{n!}   }
$$
up to scaling, the same argument as in the Sobolev case should work,
but details are left to future work.
\end{example} 
\section{Interior Superconvergence}\RSlabel{SecIntSupConv}
We now add more detail to the argument sketched at the end
of Section \RSref{SecLoc}, aiming at a proof
of superconvergence in the interior of the domain, 
if only the smoothness assumption holds, not the localization. 
\biglf
Assume a function $f\in \calh_{K*K}$ to be given, and split it into
a ``good'' and a ``bad'' part, i.e.
$$
f=g*K=g_1*K+g_2*K,\;g=g_1+g_2
$$
with $g_1$ supported in $\Omega$ and $g_2$ supported outside $\Omega$.
We would have superconvergence if we would work exclusively on
the good part $f_1=g_1*K$, by Sections \RSref{SecAbAp} and \RSref{SecLoc}.
\biglf
We focus on the bad part $f_2=g_2*K$ and want to bound it inside $\Omega$.
Assume that a ball $B_R(x)$ of radius $R$ around $x$ is still in $\Omega$.
Then we use \eref{eqf2KK} to get 
$$ 
\begin{array}{rcl}
f_2^2(x)
&\leq & \int_{\R^d\setminus \Omega} g_2^2(y)dy 
\cdot \int_{\R^d\setminus \Omega} K(x-y)^2dy\\  
&\leq & \int_{\R^d\setminus \Omega} g_2^2(y)dy 
\cdot \int_{\R^d\setminus B_R(x)} K(x-y)^2dy\\  
&=& \int_{\R^d\setminus \Omega} g_2^2(y)dy 
\cdot \int_{\R^d\setminus B_R(0)} K(y)^2dy,  
\end{array} 
$$
the second factor being a decaying function of $R$ that is independent of the 
size and placement of $\Omega$. Consequently, for each kernel $K$ there is a
radius $R$ such that the bad part of the split is not visible
within machine precision, if points have a distance of at least $R$ from the
boundary. In a somewhat sloppy form, we have
\begin{theorem}\RSlabel{theHKBE}
If there is $\calh_{K*K}$ smoothness, superconvergence 
can be always observed far enough inside the domain. If the kernel decays
exponentially towards infinity, this boundary effect
decays exponentially with the distance from the boundary.\qed 
\end{theorem}
\begin{corollary}\RSlabel{corTheHKBE}
If there is only $\calh_K$ smoothness, one can work with 
the convolution square root $\sqrt{K}$ instead of $K$, and still
get the convergence rate expected for working with $K$, but only far enough in
the interior of the domain. \qed
\end{corollary} 
For kernels with compact support,
the subdomain with superconvergence is clearly defined.  
Furthermore, this has consequences for multiscale methods that use kernels
with shrinking supports. The subdomains with superconvergence
will grow when the kernel support shrinks. 
\section{Mercer Extensions}\RSlabel{SecMExt}
The quest for functions with guaranteed superconvergence
has a simple outcome: there are complete $L-2(\Omega)$-orthonormal 
systems of those,
and they arise via Mercer expansions of kernels.
We assume a continuous translation-invariant symmetric (strictly)
positive definite Fourier-transformable kernel $K$
on $\R^d$ to be given, with ``enough'' decay at infinity.
It is reproducing in a global native space $\calh_K$ of functions on all
of $\R^d$. 
On any bounded domain $\Omega\subset\R^d$ we have a Mercer expansion
$$
K(x-y)
=\displaystyle{\sum_{n=0}^\infty \kappa_n\varphi_n(x)\varphi_n(y)=:
K_\kappa(x,y)   } 
$$ 
into orthonormal functions $\varphi_n\in L_2(\Omega)$ that are
orthogonal in the native Hilbert space $\calh(\Omega, K_\kappa)$ 
of $K_\kappa$ 
that is defined via expansions
\bql{eqfrep}
f(x)=\displaystyle{\sum_{n=0}^\infty (f,\varphi_n)_{L_2(\Omega)}
  \varphi_n(x),\;x,\,y\in \Omega  } 
\eq
and the inner product
$$
(f,g)_{\Omega,K_\kappa}:=\displaystyle{ \sum_{n=0}^\infty 
\frac{(f,\varphi_n)_{L_2(\Omega)}(g,\varphi_n)_{L_2(\Omega)}}{\kappa_n}  }
$$
such that
$$
(\varphi_j,\varphi_k)_{\Omega,K_\kappa}=\dfrac{\delta_{jk}}{\kappa_k}.
$$
It is clear that the functions $\varphi_n$ and the eigenvalues $\kappa_n$
depend on the domain $\Omega$ chosen, but we do not represent this fact in the
notation. Furthermore,
the close connection to Example \RSref{exaSeqSpe} in Section
\RSref{SecAbAp} is apparent. 
\biglf
We have to  
distinguish between the space $\calh(\Omega,K_\kappa)$ that is defined
via the expansion of $K$ into $K_\kappa$ on $\Omega$
and the space $\calh_K(\Omega)$ of Lemma \RSref{lemStaLocOrth}
in Section \RSref{SecFTS}. Since we now know that 
extensions and restrictions have to be handled carefully, 
and since the connection between local Mercer expansions and extension maps
to $\R^d$ does not seem to be treated in the literature to the required extent,
we have to proceed slowly. 
\biglf
Our first 
goal is to consider how the functions $\varphi_n$ can be extended to all of
$\R^d$, and what this means for the kernel. Furthermore,
the relation between the native spaces 
$\calh_K$, $\calh(\Omega,K_\kappa)$,
and $\calh_K(\Omega)$  is
interesting. 
\biglf
Besides the standard reproduction properties in $\calh(\Omega,K_\kappa)$, 
a Mercer
expansion  allows to write the integral operator 
\bql{eqIntOp}
(I^\Omega f)(x):=
\displaystyle{\int_\Omega K(x-y)f(y)dy =:(K*_\Omega f)(x)  } \fa x\in \Omega
\eq
as a multiplier operator
$$
f(x) \mapsto
(I^\Omega f)(x)=\displaystyle{
\sum_{n=0}^\infty \kappa_n (f,\varphi_n)_{L_2(\Omega)} \varphi_n(x)  } 
$$
with a partially defined inverse, a ``pseudodifferential''
multiplier operator
$$
f(x) \mapsto
(D^\Omega f)(x)=\displaystyle{
\sum_{n=0}^\infty \frac{(f,\varphi_n)_{L_2(\Omega)}}{\kappa_n} \varphi_n(x)  } 
$$
defined on all $f$ with
$$
\displaystyle{\sum_{n=0}^\infty 
\frac{(f,\varphi_n)^2_{L_2(\Omega)}}{\kappa^2_n}}<\infty.
$$
For such $f$, there is a local $L_2$ reproduction equation
$$
f(x)=(D^\Omega f,K(x,\cdot))_{L_2(\Omega)} 
$$
that trivially follows from 
$$
(I^\Omega f)(x)=(f, K(x,\cdot))_{L_2(\Omega)} =(K*_\Omega f)(x)
$$
and is strongly reminiscent of Taylor's formula.
The eigenvalue equation
\bql{eqeig}
\kappa_n \varphi_n(x) =\displaystyle{\int_\Omega K(x-y)  \varphi_n(y)dy
}
\fa x\in \Omega,\;n\geq 0
\eq 
can serve to extend $\varphi_n$ to all of $\R^d$. Note that we cannot use the 
norm-minimal extension in $\calh_K$ at this point,
because so far there is no connection
between these spaces. 
If we define an
{\em eigensystem extension } $\varphi_n^E$ by 
$$
\kappa_n \varphi^E_n(x) :=\displaystyle{\int_\Omega K(x-y)  \varphi_n(y)dy
}
\fa x\in \R^d,\;n\geq 0
$$
we need the decay assumption
$$
\int_\Omega K(x-y)^2dy<\infty
$$
to make the definition feasible pointwise, and if we introduce the
characteristic function $\chi_\Omega$, we can write
$$
\kappa_n \varphi^E_n= K*(\chi_\Omega \varphi_n)
$$ 
to see that $\varphi^E_n$ is well-defined as a function with Fourier transform
$$
\kappa_n (\varphi^E_n)^\wedge=K^\wedge \cdot (\chi_\Omega \varphi_n)^\wedge
=K^\wedge \cdot (\chi_\Omega \varphi_n^E)^\wedge,
$$
and it thus lies in $\calh_{K*K}$ and can be embedded into 
$\calh_K$. We note in passing that global eigenvalue
equations like the local one in 
\eref{eqeig} cannot work except in $L_2$ with the delta
``kernel'', because $\kappa_n \hat \varphi_n=\hat K \cdot 
\hat \varphi_n$
would necessarily hold.
\biglf
Anyway, from $\varphi^E_n(x)=\varphi_n(x)$ on $\Omega$ we get that
the eigenvalue equation \eref{eqeig} also holds for $\varphi_n^E$
and then for all $x\in\R^d$. Furthermore, the functions 
$\varphi^E_n$ 
satisfy the sufficient conditions for 
superconvergence, and thus they are in $\calh_K(\Omega)\cap \calh_{K*K}$.  
\biglf
We now use the notation in \eref{eqNSgen} again. 
Hitting the eigenfunction equation with $\mu_{X,a}$ yields
$$
\begin{array}{rcl}
\kappa_n \mu_{X,a}(\varphi_n^E)
& =& \displaystyle{\int_\Omega \mu_{X,a}^xK(x-y)  \varphi_n(y)dy}\\
& =& \displaystyle{\int_\Omega f_{X,a}(y)  \varphi_n(y)dy}\\
& =& (R^\Omega f_{X,a},\varphi_n)_{L_2(\Omega)}\\
&=& \kappa_n (f_{X,a},\varphi_n^E)_{K} 
\end{array}
$$
using the restriction map $R^\Omega$. 
Since all parts are continuous on ${\calh_K}$,
this generalizes to
\bql{eqomegaconn}
(R^\Omega f,\varphi_n)_{L_2(\Omega)}=\kappa_n (f,\varphi_n^E)_K 
\fa f\in{\calh(K,\R^d)}
\eq
and in particular 
$$
\delta_{jk}=\kappa_k (\varphi_j^E,\varphi_k^E)_K,
\;j,k\geq 0
$$
proving that the $\calh(\Omega,K_\kappa)$-orthogonality of the $\varphi_n$
carries over to the same orthogonality of the $\varphi_n^E$
in $\calh_K$, though the spaces and norms are defined differently.
Another consequence of \eref{eqomegaconn} combined with Lemma
 \RSref{lemStaLocOrth} is 
\begin{lemma}\RSlabel{lemHKOsame}
The subspace $\calh_K(\Omega)$ is the $\calh_K$-closure of the span
of the $\varphi_n^E$. \qed
\end{lemma} 
The extension via the eigensystems generalizes \eref{eqfrep} to 
\bql{eqfE}
f^E(x):=\displaystyle{\sum_{n=0}^\infty (f,\varphi_n)_{L_2(\Omega)} \varphi^E_n(x)  }
\fa x\in \R^d.
\eq
\begin{lemma}\RSlabel{lemExt}
The extension map $f\mapsto f^E$ is isometric as a map from
$\calh(\Omega,K_\kappa)$ to ${\calh_K(\Omega)}$. \qed
\end{lemma} 
It is now natural to define a kernel
$$
K^E(x,y):=\displaystyle{\sum_{n=0}^\infty 
\kappa_n\varphi_n^E(x)\varphi_n^E(y)   } 
$$
that coincides with $K$ on $\Omega\times \Omega$. If we insert it into
\eref{eqomegaconn}, we get
$$
\begin{array}{rcl}
\kappa_n (K^E(x,y),\varphi_n^E)_K 
&=&
((R^\Omega)^y K^E(x,y),\varphi_n)_{L_2(\Omega)}\\
&=&
\displaystyle{\left(\sum_{k=0}^\infty \kappa_k\varphi_k^E(x)\varphi_k,
\varphi_n\right)_{L_2(\Omega)}   }\\ 
&=&
\kappa_n\varphi_n^E(x)
\end{array}
$$
proving that $K^E$ is reproducing on the span of the $\varphi^E$
in the inner product of ${\calh_K}$, i.e. on $\calh_K(\Omega)$, 
and the actions of $K$ and $K^E$ on that subspace are the same. 
\begin{theorem}\RSlabel{theSumMerc}
The localized spaces $\calh_K(\Omega)$ and $\calh(\Omega,K_\kappa)$ 
can be identified, and the extensions to $\R^d$ via eigenfunctions
and by norm-minimality coincide. Working with a Mercer expansion on $\Omega$ 
means working in the space $\calh_K(\Omega)$ that shows superconvergence
if $\calh_{K*K}$-smoothness is added.   
\end{theorem} 
A similar viewpoint connected to Mercer expansions is that
superconvergence occurs whenever there is a {\em range condition}
in the sense of Integral Equations,  i.e. the given function $f$
is in the range of the integral operator \eref{eqIntOp}.
\section{Numerical Examples}\RSlabel{SecExa}
The reproducing kernels of $W_2^1(\R^1)$ and $W_2^2(\R^1)$ are
$$
\begin{array}{rcl}
K_1(r)&:=&\sqrt{\frac{\pi}{2}}\exp(-r),\\
K_2(r)&:=&\sqrt{\frac{\pi}{2}}\exp(-r)(1+r),\\
\end{array}
$$
respectively, and we shall mainly 
work with $K:=K_2$ in $\calh_K=W_2^2(\R^1)$,
continuing Example \RSref{exaW21} from Section \RSref{SecSob}.
We use the function 
$f:=K_2*\chi_{[-1,+1]}$, which can easily be calculated explicitly as
$$
\begin{array}{rcl}
f(x)=\left\{\begin{array}{lclcc}
e^{+x-1}(x-3)&+& e^{+x+1}(1-x) & & x\leq -1\\
e^{+x-1}(x-3) &-&e^{-1-x}(x+3)+4 & & -1\leq x\leq +1\\
e^{-x+1}(1+x)&-&e^{-1-x}(x+3) & & 1\leq x\\
\end{array}     \right\}
\end{array} 
$$
with the correct extension to $\R$ by solutions
of $L_4f=(f-f'')-(f-f'')''=0$ on either side, together with the
needed decay at infinity.
\biglf
The convolution domain $[-1,+1]$
is kept fixed, but then we vary the domain $\Omega=[-C,+C]$ 
that we work on. 
Note that reasonable solutions will try to come up  
with coefficients that are 
a discretization of the characteristic function  $\chi_{[-1,+1]}$,
but this is not directly possible for $C<1$. 
\biglf
In each domain chosen, we took equidistant interpolation points,
and for estimating $L_2$ norms, we calculated 
a root-mean-square error on a sufficiently fine subset.
Working in $W_2^2(\R^1)$ with the kernel $K_2$ would usually
give a global $L_2$ interpolation error of order $h^2$ due to standard
results, see e.g. \RScite{wendland:2005-1}, and this is the order arising in
the standard sampling inequality  
that is doubled
by Theorem \RSref{theSupConvLoc}. 
Thus we expect a convergence rate of $h^4$
in the superconvergence situation, while the normal rate is $h^2$.  
\biglf
If we use $C=1.2$ and interpolate $f_2$ in $\calh(K,\R^d)$ there.
we are in the superconvergence 
case, because $f_2$ is a convolution with $K$ of a function supported
in $[-1,+1]\subset\Omega$. The observed rates are 
around $4$ in $[-1.2,+1.2]$ and in the ``interior'' domain $[-0.8,+0.8]$,
see Figure \RSref{figC1p2}. Up to a Gibbs phenomenon,
the interpolant recovers $\chi_{[-1,+1]}$, and this is also visible when looking
at the error.  

\begin{figure}[hbt] 
\begin{center}
\includegraphics[width=6cm,height=6cm]{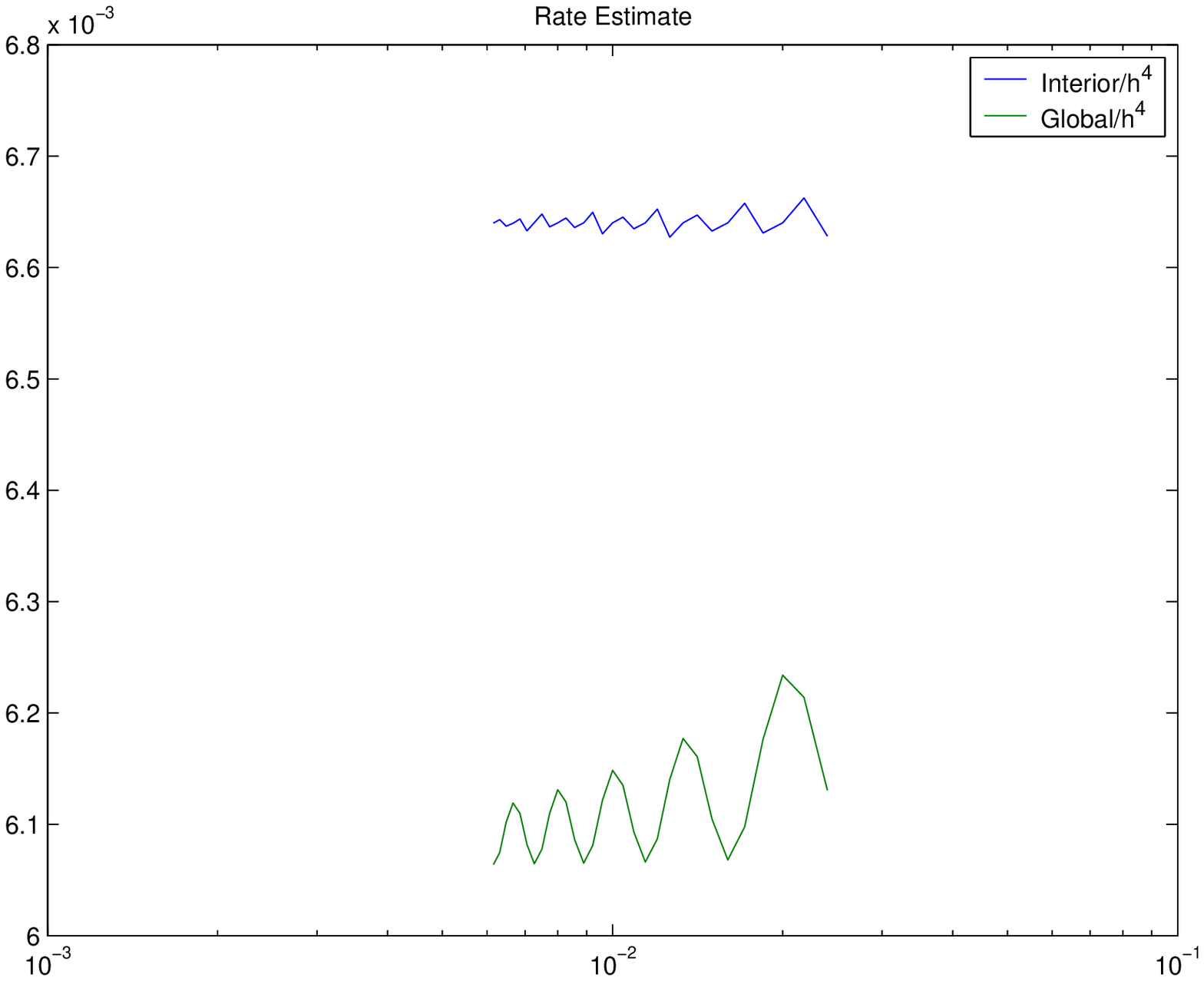}
\includegraphics[width=6cm,height=6cm]{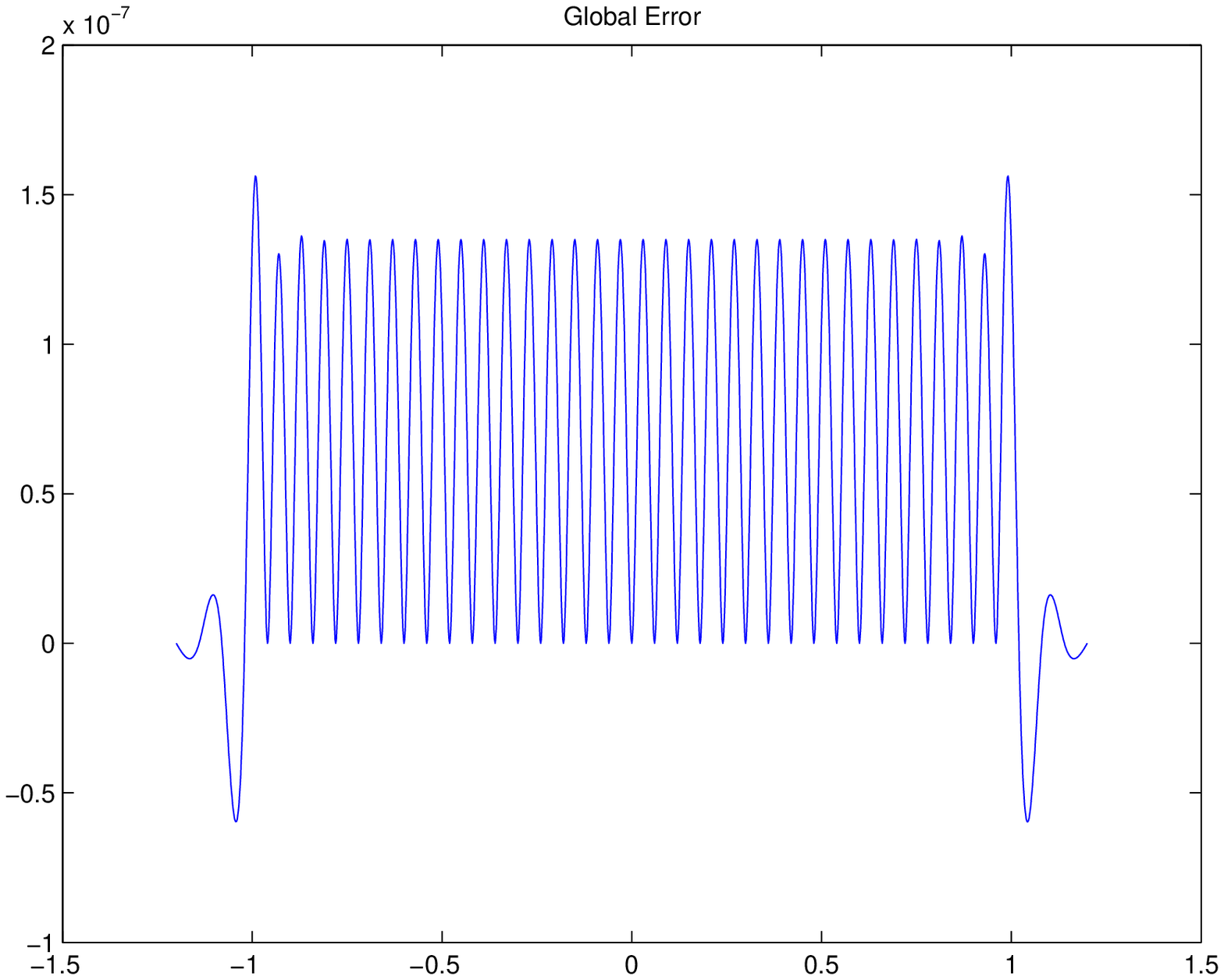}
\end{center}
\caption{Superconvergence case in $[-1.2,+1.2]$, rate estimates
(left) and error function for 41 points (right)
\RSlabel{figC1p2}}
\end{figure}

For $C=0.8$, we still have enough smoothness for
superconvergence, but the localization condition \eref{eqlocalizgen} fails.
The standard 
expected global convergence rate is 2, but in the ``interior''
$[-0.6,+0.6]$ we still see superconvergence of order 4 
in Figure \RSref{figC0p8}. The global error is attained at the boundary.
\biglf
 
\begin{figure}[hbt] 
\begin{center}
\includegraphics[width=6cm,height=6cm]{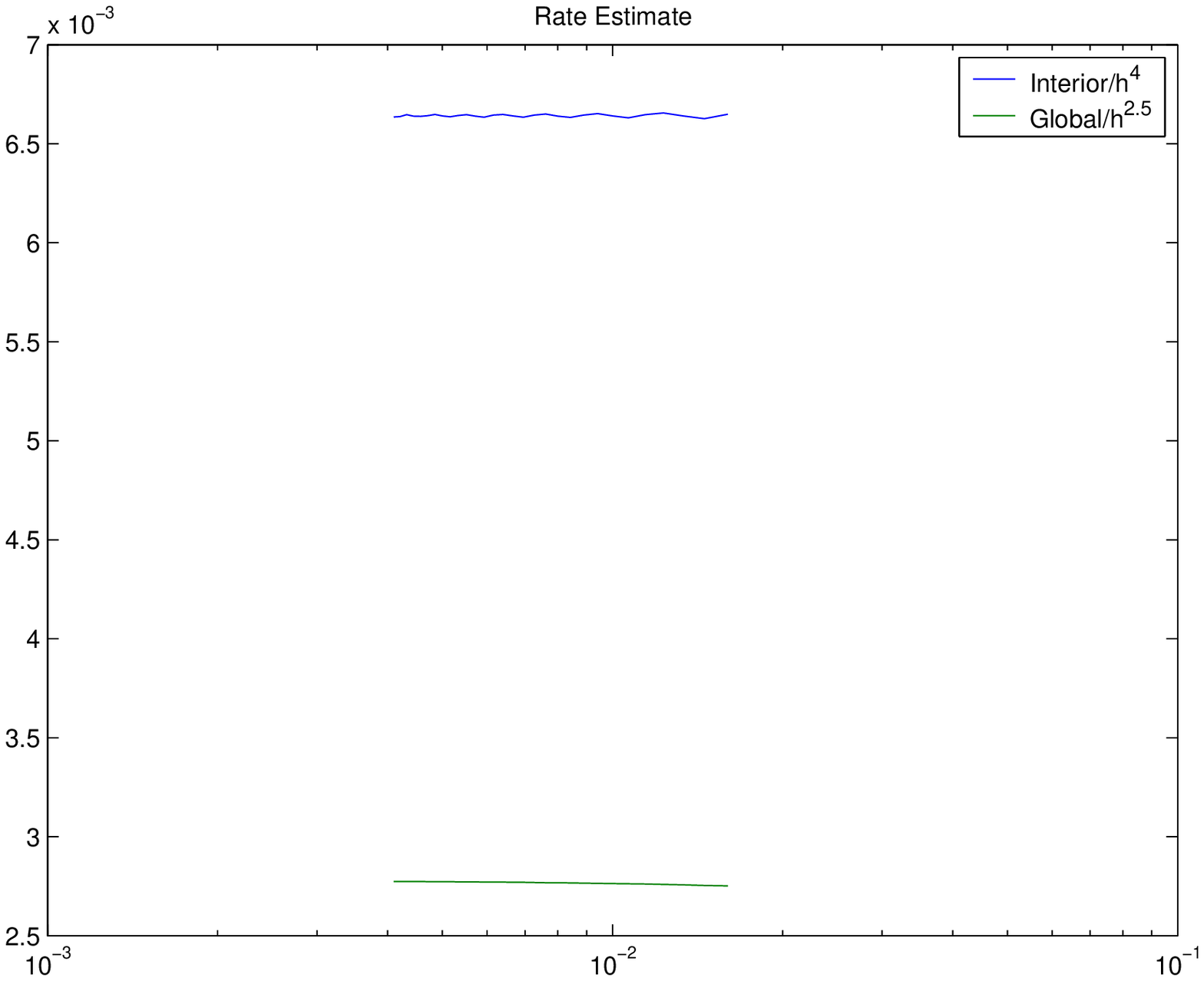}
\includegraphics[width=6cm,height=6cm]{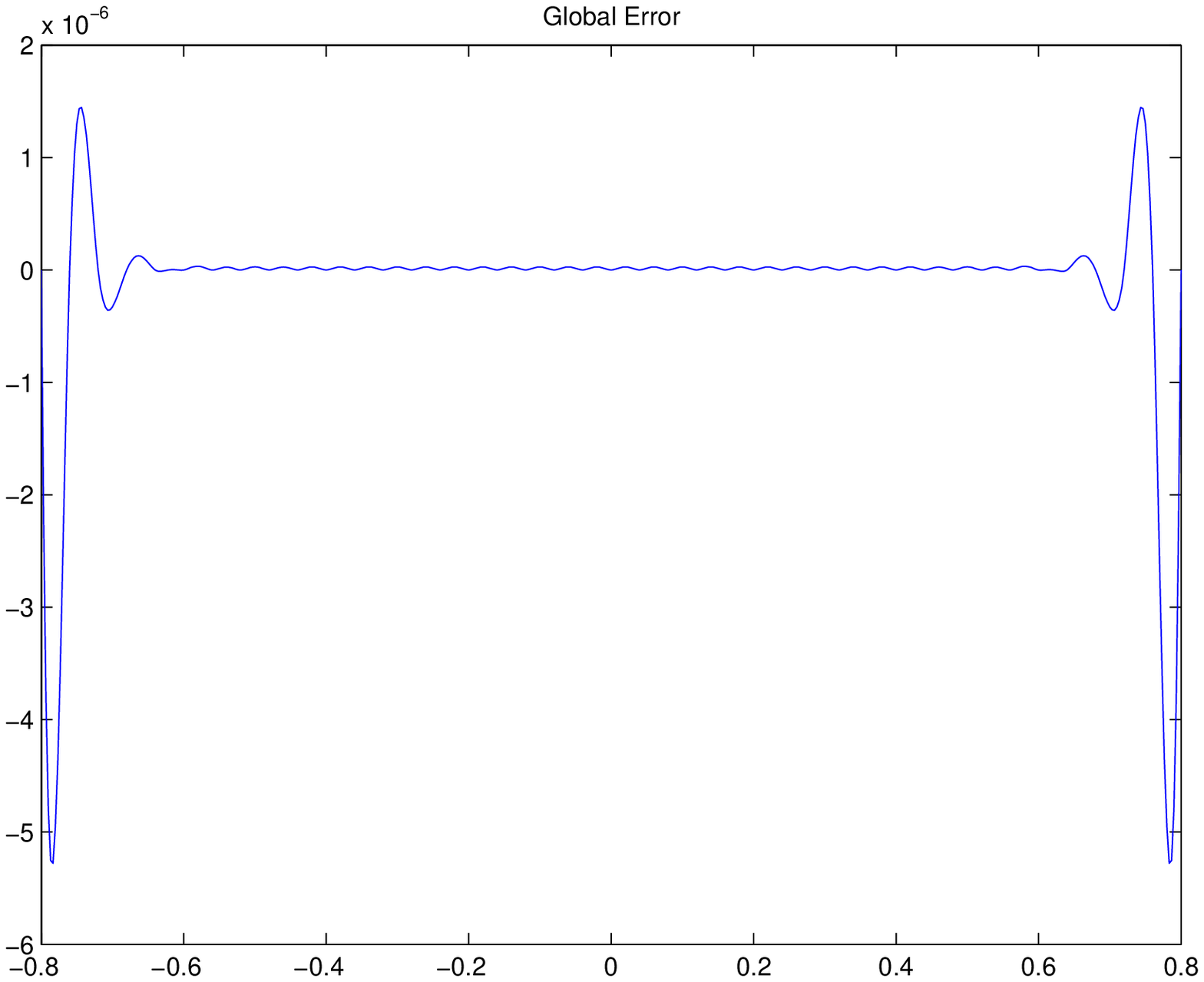}
\end{center}
\caption{Convergence in $\Omega:=[-0.8,+0.8]$ and ``interior''
$[-0.6,+0.6]$, rate estimates
(left) and error function for 41 points (right)
\RSlabel{figC0p8}}
\end{figure}

Surprisingly, the global rate is 2.5 instead of 2, and this is confirmed for
many other cases, even various ones with just $W_2^2(\R^1)$ smoothness.
This is another instance of superconvergence, and it needs further work.
Experimentally, it can be observed that the
norms $\|f-s_{f,X,K}\|_K$ often go to zero like $1/\sqrt{|X|}$, 
possibly accounting 
for the extra
$\sqrt{h}$ contribution to the usual convergence rate $2$ that is obtained
when assuming that the norms are only bounded by $\|f\|_K$.
\biglf
The standard error analysis of kernel-based interpolation 
of functions $f\in \calh_K(\Omega)$ using a kernel $K$ and a set $X$ of nodes
ignores the fact that the Hilbert space error $\|f-s_{f,X,K}\|_K$
decreases to zero when $|X|$ gets large and finally ``fills'' the domain.
It seems to be a long-standing problem to turn this obvious
fact into a convergence rate that is better than the usual
one given by sampling inequalities that just use the upper bound
$\|f\|_K$ for that error. 
\bibliographystyle{plain}

\end{document}